\nonstopmode
\documentclass[]{amsart}
\usepackage{verbatim}
\usepackage{amssymb,bbm}
\usepackage{graphicx}
\usepackage{psfrag}
\usepackage{ifpdf}
\usepackage{framed}
\usepackage{multirow}
\usepackage{todonotes}
\ifpdf
\usepackage[notightpage]{pst-pdf}
\fi

\usepackage[all]{xy}
\newdir{ >}{{}*!/-10pt/@{>}}
\newdir^{ (}{{}*!/-5pt/@^{(}}
\newdir_{ (}{{}*!/-5pt/@_{(}}
\def\cgaps#1{}
\def\Cgaps#1{}
\def\AMStag#1{}
\def\AMSunderset#1\to#2{\underset{#1}{#2}}
\def\AMSoverset#1\to#2{\overset{#1}{#2}}
\def\undersetbrace#1\to#2{\underbrace{#2}_{#1}}
\def\oversetbrace#1\to#2{\overbrace{#2}^{#1}}
\def\3{\ss}

\swapnumbers

\newtheorem*{proposition*}{Proposition}
\newtheorem{theorem}[subsection]{Theorem}
\newtheorem*{theorem*}{Theorem}
\newtheorem{lemma}[subsection]{Lemma}
\newtheorem*{lemma*}{Lemma}

\newtheorem*{corollary*}{Corollary}

\newtheorem*{result*}{Result}

\def\cit#1#2{\ifx#1!\cite{2}\else#2\fi} 
\def\ign#1{}             
\def\o{\circ\,} 
\def\X{\mathfrak X} 
\def\al{\alpha} 
\def\be{\beta}

\def\ep{\varepsilon}

\def\la{\lambda} 
\def\rh{\rho} 
 
\def\ta{\tau} 
\def\ph{\varphi} 
 
\def\ps{\psi} 
 
\def\Ga{\Gamma}

\def\i{^{-1}}

\def\F{\mathcal{F}}

\def\exp{\operatorname{exp}}
\let\on=\operatorname

\DeclareMathOperator{\Diff}{Diff}

\newcommand{\R}{\mathbb{R}}

\newcommand{\ud}{\,\mathrm{d}}

\newcommand{\one}{\mathbbm{1}}

\title[Fractional order Sobolev 
metrics on the diffeomorphism group. II]
{Geodesic distance for right invariant Sobolev 
metrics of fractional order on the diffeomorphism group. II}
\author{Martin Bauer, Martins Bruveris, Peter W. Michor}
 
\address{
Martin Bauer, Peter W.\ Michor: Fakult\"at f\"ur Mathematik,
Universit\"at Wien, 
Nordbergstrasse 15, A-1090 Wien, Austria.\newline\indent
Martins Bruveris: Insitut de math\'ematiques, EPFL, CH-1015, Lausanne, Switzerland.
}
\email{bauer.martin@univie.ac.at}
\email{martins.bruveris@epfl.ch}
\email{peter.michor@esi.ac.at}
\thanks{Martin Bauer was  supported by `Fonds zur
F\"orderung der wissenschaftlichen                    
Forschung, Projekt P~24625'.}
\begin{document}
\begin{abstract}
The geodesic distance vanishes on the group $\Diff_c(M)$ 
of compactly supported diffeomorphisms of a Riemannian manifold $M$ of bounded geometry, 
for the right invariant weak Riemannian metric which is induced by the Sobolev metric 
$H^s$ of order $0\le s<\tfrac12$ on the Lie algebra $\X_c(M)$ of vector fields with compact 
support. 
\end{abstract} 

\maketitle

\section{Introduction}

In the article \cite{Bauer2011c_preprint} we studied right invariant metrics on the group $\Diff_c(M)$ 
of compactly supported diffeomorphisms of a manifold $M$, which are induced by the Sobolev metric 
$H^s$ of order $s$ on the Lie algebra $\X_c(M)$ of vector fields with compact 
support. 
We showed that for $M=S^1$ the geodesic distance on $\Diff(S^1)$ vanishes if and only if $s\leq\frac12$. 
For other manifolds, we showed that the geodesic distance on $\Diff_c(M)$ vanishes for $M=\R\times N, s<\frac12$ and for 
$M=S^1\times N, s\leq\frac12$, with $N$ being a compact Riemannian manifold.

Now we are able to complement this result by:
{\em The geodesic distance vanishes on $\Diff_c(M)$ for any Riemannian manifold $M$ of bounded geometry, 
if $0\le s<\frac12$.}

We believe that this result holds also for $s=\frac12$, but we were able to overcome the technical 
difficulties only for the manifold $M=S^1$, in \cite{Bauer2011c_preprint}. We also believe that it 
is true for the regular groups $\Diff_{\mathcal H^\infty}(\mathbb R^n)$ and 
$\Diff_{\mathcal S}(\mathbb R^n)$ as treated in \cite{Michor2012b_preprint}, and for all Virasoro 
groups, where we could prove it only for $s=0$ in \cite{Bauer2012c}.

In Section \ref{sec:frac}, we review the definitions for Sobolev norms of fractional 
orders on diffeomorphism groups as presented in \cite{Bauer2011c_preprint}
and extend them to diffeomorphism groups of manifolds of bounded geometry.
Section \ref{vanishing} is devoted to the main result.

\section{Sobolev metrics $H^s$ with $s\in \mathbb R$}
\label{sec:frac}

\subsection{Sobolev metrics $H^s$ on $\mathbb R^n$}
For $s\geq 0$ the Sobolev $H^s$-norm of an $\R^n$-valued function $f$ on $\R^n$ is
defined as
\begin{equation}
\label{Hs_norm}
\| f \|_{H^s(\R^n)}^2 = \| \F\i (1+|\xi|^2)^{\frac{s}2} \F f \|_{L^2(\R^n)}^2\;,
\end{equation}
where $\F$ is the Fourier transform 
\[ \F f(\xi) = (2\pi)^{-\frac n 2} \int_{\R^n} e^{-i \langle  x,\xi\rangle} f(x)
\ud x\;, \]
and $\xi$ is the independent variable in the frequency domain.
An equivalent norm is given by
\begin{equation}
\label{Hsbar_norm}
\| f \|^2_{\overline{H}^s(\R^n)} = \| f \|^2_{L^2(\R^n)} + \| |\xi|^s \F f
\|^2_{L^2(\R^n)}\;.
\end{equation}
The fact that both norms are equivalent is based on the inequality
\[ \frac{1}C \big( 1 + \sum_j |\xi_j|^s \big) \leq \big(1 + \sum_j |\xi_j|^2
\big)^{\frac{s}{2}} \leq C \big( 1 + \sum_j |\xi_j|^s \big)\;, \]
holding for some constant $C$. For $s>1$ this says that all $\ell^s$-norms on
$\R^{n+1}$ are equivalent. 
But the inequality is true also for $0 < s < 1$, even though the expression does
not define a norm on $\R^{n+1}$.
Using any of these norms we obtain the  Sobolev spaces with non-integral $s$
$$H^s(\R^n) = \{ f \in L^2(\R^n): \| f \|_{H^s(\R^n)} < \infty \}\;.$$
We will use the second version of the norm in the proof of the theorem, since it will make calculations easier.

\subsection{Sobolev metrics for Riemannian manifolds of bounded geometry}
\label{boundedGeometry}
Following \cite[Section~7.2.1]{Triebel1992} we will now introduce the spaces
$H^s(M)$ on a manifold 
$M$. 
If $M$ is not compact we equip $M$ with a Riemannian 
metric $g$ of bounded geometry which exists by \cite{Greene1978}. This means that
\\\indent $(I)$\qquad 
The injectivity radius of $(M,g)$ is positive.
\\\indent $(B_\infty)$\quad Each iterated covariant derivative of the curvature 
\\\hphantom{A}\qquad\qquad is uniformly $g$-bounded: 
        $\|\nabla^i R\|_g<C_i$ for $i=0,1,2,\dots$.
\\
The following is a compilation of special cases of results collected in 
\cite[Chapter 1]{Eichhorn2007}, 
who treats Sobolev spaces only for integral order.

\begin{proposition*} [\cite{Kordyukov1991}, \cite{Shubin1992}, \cite{Eichhorn1991}]
If $(M,g)$ satisfies $(I)$ and $(B_\infty)$ then the following holds:
\begin{enumerate}
	\item $(M,g)$ is complete.
	\item There exists $\ep_0>0$ such that for each $\ep\in (0,\ep_0)$ there
is a countable cover of 
        $M$ by geodesic balls $B_\ep(x_\al)$ such that the cover of $M$ by the
balls $B_{2\ep}(x_\al)$ 
        is still uniformly locally finite.  
	\item Moreover, there exists a partition of unity $1= \sum_\al \rh_\al$
on $M$ such that $\rh_\al\ge 0$, 
        $\rh_\al\in C^\infty_c(M)$, $\on{supp}(\rh_\al)\subset B_{2\ep}(x_\al)$,
and 
				$|D_u^\be \rh_\al|<C_\be$ where $u$ are normal
(Riemann exponential) coordinates in 
        $B_{2\ep}(x_\al)$. 
	\item In each $B_{2\ep}(x_\al)$, in normal coordinates, we have 
	      $|D_u^\be g_{ij}|<C'_\be$,  
	      $|D_u^\be g^{ij}|<C''_\be$, and 
	      $|D_u^\be \Ga^m_{ij}|<C'''_\be$,
				where all constants are independent of $\al$.  
\end{enumerate}
\end{proposition*}
We can now define 
the $H^s$-norm of a function $f$ on $M$:
\begin{align*}
\| f \|_{H^s(M,g)}^2 &= \sum_{\al=0}^\infty \| (\rh_\al f)\o \exp_{x_\al}
\|^2_{H^s(\R^n)} =
\\&
= \sum_{\al=0}^\infty \| \F\i (1+|\xi|^2)^{\frac{s}2} \F ((\rh_\al f)\o
\exp_{x_\al}) \|^2_{L^2(\R^n)}\;.
\end{align*}
If $M$ is compact the sum is finite.
Changing the charts or the partition of unity leads to equivalent norms by the
proposition above, 
see \cite[Theorem 7.2.3]{Triebel1992}.
For integer $s$ we get norms which are equivalent to the Sobolev norms treated
in 
\cite[Chapter 2]{Eichhorn2007}. The norms depends on the choice of the Riemann
metric $g$. This 
dependence is worked out in detail in \cite{Eichhorn2007}.

For vector fields we use the trivialization of the tangent bundle 
that is induced by the coordinate charts and define the norm in each coordinate
as above. 
This leads to a (up to equivalence) well-defined $H^s$-norm on the Lie algebra
$\X_c(M)$.

\subsection{Sobolev metrics on $\Diff_c(M)$}
Given a norm on $\X_c(M)$ we can use the right-multiplication in the
diffeomorphism group 
$\Diff_c(M)$  to extend this norm to a right-invariant Riemannian metric on
$\Diff_c(M)$. 
In detail, given $\ph \in \Diff_c(M)$ and $X, Y \in T_\ph \Diff_c(M)$ we define
\[ G^s_\ph(X,Y) = \langle X\o \ph\i, Y\o\ph\i \rangle_{H^s(M)}\;.\]

We 
are interested solely in questions of vanishing and non-vanishing of geodesic
distance. 
These properties are invariant under changes to equivalent inner products, since
equivalent inner products on the Lie Algebra
\[ \frac{1}C \langle X, Y \rangle_1 \leq \langle X, Y \rangle_2 \leq C \langle
X, Y \rangle_1 \]
imply that the geodesic distances will be equivalent metrics
\[ \frac{1}C \on{dist}_1(\ph, \ps) \leq \on{dist}_2(\ph, \ps) \leq C
\on{dist}_1(\ph, \ps) \;.\]
Therefore the ambiguity -- dependence on the charts and the partition of unity -- in the definition of the $H^s$-norm is of no concern to us.

\section{Vanishing geodesic distance}\label{vanishing}

\begin{theorem}[Vanishing geodesic distance]\label{vanishingthm}
The Sobolev metric of order $s$ induces vanishing geodesic distance on $\on{Diff}_c(M)$ if:
\begin{itemize}
\item
$0\leq s < \frac 12$ and $M$ is any Riemannian manifold of bounded geometry.
\end{itemize}
This means that any two diffeomorphisms in the same connected component of
$\Diff_c(M)$ 
can be connected by a path of arbitrarily short $G^s$-length.
\end{theorem}


In the proof of the theorem we shall make use of the following lemma from \cite{Bauer2011c_preprint}.
\begin{lemma}[{\cite[Lemma~3.2]{Bauer2011c_preprint}}]\label{VectorfieldR}
Let $\ph \in \Diff_c(\R)$ be a diffeomorphism satisfying $\ph(x) \geq x$ and let $T>0$ be fixed.  
Then for each $0 \leq s < \tfrac12$ and $\ep>0$ there exists a time dependent vector field 
$u_{\R}^{\ep}$ of the form
$$u_{\R}^{\ep}(t,x)=\one_{[f^{\ep}(t), g^{\ep}(t)]}*G_{\ep}(x),$$
with $f, g \in C^\infty([0,T])$, such that  its  flow  $\ph^{\ep}(t,x)$  satisfies -- independently of $\ep$ -- the properties $\ph^{\ep}(0,x) = x$,
$\ph^{\ep}(T,x)=\ph(x)$ and whose $H^s$-length is 
smaller than $\ep$, i.e.,
$$\on{Len}(\ph{^\ep})=\int_0^T \|u_{\R}^{\ep}(t,\cdot)\|_{H^s}dt\leq C\|f^\ep-g^\ep\|_{\infty}\leq  \ep\;.$$

Furthermore $\{ t \,:\, f^\ep(t) < g^\ep(t)\} \subseteq \on{supp}(\ph)$ and there exists a limit function $h\in C^\infty([0,T])$, such that $f^\ep \to h$ and $g^\ep \to h$ for $\ep \to 0$ and the convergence is uniform in $t$.
\end{lemma}

Here, $G_\ep (x) = \frac1{\ep}G_1(\frac{x}{\ep})$ is a smoothing kernel, defined via a smooth bump function
$G_1$ with compact support.

\begin{proof}[Proof of Theorem~\ref{vanishingthm}]
Consider the connected component $\on{Diff}_0(M)$ of $\on{Id}$,
i.e. those diffeomorphisms of $\on{Diff}_c(M)$, for which there exists at
least one path, joining them to the identity. Denote by $\on{Diff}_c(M)^{L=0}$ 
the set of all diffeomorphisms $\ph$ that
can be reached from the identity by curves of arbitrarily short length,
i.e., for each $\ep>0$ there 
exists a curve from $\on{Id}$ to $\ph$ with length smaller than $\ep$. 
\\{\bf Claim A.} \emph{$\Diff_c(M)^{L=0}$ is a normal subgroup of $\Diff_0(M)$.}
\\{\bf Claim B.} \emph{$\Diff_c(M)^{L=0}$ is a non-trivial subgroup of $\Diff_0(M)$.}

By \cite{Thurston1974} or \cite{Mather1974}, the group $\Diff_0(M)$ is simple. 
Thus claims A and B imply $\on{Diff}_c(M)^{L=0}=\Diff_0(M)$, which proves the theorem.

The proof of claim A can be found in \cite[Theorem 3.1]{Bauer2011c_preprint} and works without change in the case of $M$ being an arbitrary manifold and hence we will not repeat it here. It remains to show that $\on{Diff}_c(M)^{L=0}$ contains a diffeomorphism $\ph \neq \on{Id}$.

We shall first prove claim B for $M=\R^n$ and then show how to extend the arguments to arbitrary manifolds. Choose a diffeomorphism $\ph_\R \in \on{Diff}_c(\R)$ with $\on{supp}(\ph_\R) \subseteq [1, \infty)$ and let
$$u^{\ep}_{\R}(t,x):=\one_{[f^{\ep}(t), g^{\ep}(t)]}*G_{\ep}(x)$$
be the family of vector fields constructed in Lemma \ref{VectorfieldR}, whose flows at time $T$ equal $\ph_\R$. We extend the vector field $u^\ep_\R$ to a vector field $u^\ep_{\R^n}$ on $\R^n$ via
$$u^\ep_{\R^n}(x_1,\ldots,x_n):= \big(u^\ep_{\R}(|x|),0,\ldots,0\big)\;.$$
The flow of this vector field is given by
\[
\ph_{\R^n}^\ep(t, x_1,\dots,x_n) = \left( \ph_\R^\ep(t, |x|), x_2, \dots, x_n \right)\;,
\]
where $\ph_\R^\ep$ is the flow of $u_\R^\ep$. In particular we see that at time $t=T$
\[
\ph_{\R^n}^\ep(t, x_1,\dots,x_n) = \left( \ph_\R(|x|), x_2, \dots, x_n \right)\;,
\]
the flow is independent of $\ep$. So it remains to show that for the length of the path $\ph_{\R^n}^\ep(t,\cdot)$ we have
\[
\on{Len}(\ph_{\R^n}^\ep) \to 0 \quad \text{as} \quad \ep \to 0\;.
\]
We can estimate the length of this path via
\begin{align*}
\on{Len}(\ph^{\ep}_{\R^n})^2 &= \left(\int_0^T \| u^{\ep}_{\R^n}(t,.)\|_{H^s(\R^n)} \ud t\right)^2 \leq T \int_0^T
\| u^{\ep}_{\R^n}(t,.)\|^2_{H^s({\R^n})} \ud t\\&=T \int_0^T
\| u^{\ep}_{\R}(t,|\cdot|)\|^2_{H^s({\R^n})} \ud t=T \int_0^T
\| \one_{[f^{\ep}(t), g^{\ep}(t)]}*G_{\ep}(|x|)\|^2_{H^s({\R^n})} \ud t\\
&\leq C(G_1, T) \int_0^T
\| \one_{[f^{\ep}(t), g^{\ep}(t)]}(|\cdot|)\|^2_{H^s({\R^n})} \ud t\;,
\end{align*}
where the last estimate follows from
\begin{align*}
\|\one_{[f^{\ep}(t), g^{\ep}(t)]}&*G_{\ep}(|x|)\|^2_{H^s({\R^n})}= \\
&= \int_{\R^n} (1+|\xi|^{2s}) \left[\F\left(\one_{[f^{\ep}(t), g^{\ep}(t)]}(|\cdot|)\right)(\xi)\right]^2\,
\left[\F\left(G_\ep(|\cdot|)\right)(\xi)\right]^2 \ud \xi \\
&= \int_{\R^n} (1+|\xi|^{2s}) \left[\F\left(\one_{[f^{\ep}(t), g^{\ep}(t)]}(|\cdot|)\right)(\xi)\right]^2\,
\left[\F\left(G_1(|\cdot|)\right)(\ep \xi)\right]^2 \ud \xi \\
&\leq \left\| \F G_1(|\cdot|) \right\|_{L^\infty}^2\cdot
\|\one_{[f^{\ep}(t), g^{\ep}(t)]}(|\cdot|)\|^2_{H^s({\R^n})}\;.
\end{align*}
Hence it is sufficient to show that
\[
\left\|\one_{[f^{\ep}(t), g^{\ep}(t)]}(|\cdot|)\right\|_{H^s({\R^n})} \on \to 0 \quad \text{as} \quad \ep \to 0 \quad \text{uniformly in } t\;.
\]

To compute the $H^s$-norm of $\one_{[f^{\ep}(t), g^{\ep}(t)]}(|\cdot|)$ we first Fourier-transform it. The Fourier-transform of a radially symmetric function $v(|\cdot|) \in L^1(\R^n)$ is again radially symmetric and given by the following formula, see \cite[Theorem 3.3]{Stein1971},
\[
(\F v(|\cdot|))(\xi) = 2\pi |\xi|^{1-n/2} \int_0^{\infty} J_{n/2-1}(2\pi |\xi| s) v(s) s^{n/2} \ud s\;,
\]
with $J_{n/2-1}$ denoting the Bessel function of order $\tfrac n2 -1$. To simplify notation we will 
omit the dependece of the vector field $\one_{[f^{\ep}(t), g^{\ep}(t)]}(|\cdot|)$ on $t$ and $\ep$.  
Changing coordinates, this becomes 
\[
(\F \one_{[f,g]}(|\cdot|))(\xi) = (2\pi)^{-n/2} |\xi|^{-n} \int_{2\pi f |\xi|}^{2\pi g |\xi|} J_{n/2-1}(s) s^{n/2} \ud s\;.
\]
This integral can be evaluated explicitly using the following integral identity for Bessel functions 
from \cite[(10.22.1)]{NIST2010}
\[
\int z^{\nu+1} J_\nu(z) \ud z = z^{\nu+1} J_{\nu+1}(z),\quad \nu \neq -\tfrac 12\;.
\]
This gives us
\[
(\F \one_{[f,g]}(|\cdot|))(\xi) = |\xi|^{-n/2} \left( J_{n/2}(2\pi g |\xi|) g^{n/2} - J_{n/2}(2\pi f |\xi|) f^{n/2}\right)\;.
\]
The $H^s$-norm of $\one_{[f,g]}(|\cdot|)$ is given by
\[
\left\| \one_{[f,g]}(|\cdot|)\right\|_{H^s(\R^n)}^2 = \int_{\R^n} \left(1 + |\xi|^{2s}\right) \F \one_{[f,g]}(|\cdot|)(\xi)^2 \ud \xi\;.
\]
We will only consider the term involving $|\xi|^{2s}$, since the $L^2$-term can be estimated in the same way by setting $s=0$. Transforming to polar coordinates we obtain
\begin{align*}
\int_{\R^n} |\xi|^{2s} \big(\F &\one_{[f,g]}(|\cdot|)(\xi)\big)^2 \ud \xi  =
\\&
= \int_{\R^n} |\xi|^{2s-n}  \left( J_{n/2}(2\pi g |\xi|) g^{n/2} - J_{n/2}(2\pi f |\xi|) f^{n/2}\right)^2 \ud \xi \\
&= \on{Vol}(S^{n-1}) \int_0^\infty r^{2s-1} \left( J_{n/2}(2\pi g r) g^{n/2} - J_{n/2}(2\pi f r) f^{n/2}\right)^2 \ud r\,.
\end{align*}
The above integral is non-zero only for those $t$, where $f^\ep(t) \neq g^\ep(t)$. From Lemma \ref{VectorfieldR} and our assumptions on $\ph_\R$ we know that
\[
\{t \,:\, f^\ep(t) < g^\ep(t)\} \subseteq \on{supp}(\ph_\R) \subseteq [1,\infty)\;.
\]
Thus both $f^\ep(t)$ and $g^\ep(t)$ are different and away from 0 and we can evaluate the above integral using the identity \cite[(10.22.57)]{NIST2010},
\begin{align*}
\int_0^\infty \frac{J_\mu(at)J_\nu(at)}{t^\la} \ud t =
\frac{\left(\tfrac 12 a\right)^{\la - 1} \Ga\left(\tfrac \mu 2 + \tfrac \nu 2 - \tfrac \la 2 + \tfrac 12\right) \Ga\left(\la\right)} {2 \Ga\left(\tfrac \la 2 + \tfrac \nu 2 - \tfrac \mu 2 + \tfrac 12\right) \Ga\left(\tfrac \la 2 + \tfrac \mu 2 - \tfrac \nu 2 + \tfrac 12\right) \Ga\left(\tfrac \la 2 + \tfrac \mu 2 + \tfrac \nu 2 + \tfrac 12\right)}\,,
\end{align*}
which holds for $\on{Re}(\mu+\nu+1) > \on{Re} \la > 0$ and the identity \cite[(10.22.56)]{NIST2010},
\begin{align*}
\int_0^\infty &\frac{J_\mu(at)J_\nu(bt)}{t^\la} \ud t =
\\&
=\frac{a^\mu \Ga\left(\tfrac \nu 2 + \tfrac \mu 2 - \tfrac \la 2 + \tfrac 12\right)} {2^\la b^{\mu-\la+1} \Ga\left(\tfrac \nu 2 - \tfrac \mu 2 + \tfrac \la 2 + \tfrac 12\right)} \mathbf{F}\left(\tfrac \nu 2 + \tfrac \mu 2 - \tfrac \la 2 + \tfrac 12, \tfrac \mu 2 - \tfrac \nu 2 - \tfrac \la 2 + \tfrac 12; \mu+1; \tfrac{a^2}{b^2} \right)\,,
\end{align*}
which holds for $0 < a < b$ and $\on{Re}(\mu+\nu+1) > \on{Re} \la > -1$. Here $\mathbf{F}(a, b;c;d)$ is the regularized hypergeometric function.
Using these identities with $\la=1-2s$, $\mu=\nu=\tfrac n2$, $a = 2\pi f$ and $b = 2\pi g$ we obtain
\begin{align*}
\int_0^\infty r^{2s-1} J_{n/2}(2\pi f r)^2 \ud r &= \frac 12 (\pi f)^{-2s} \frac{\Ga\left(\tfrac n2 + s \right) \Ga(1-2s)}{\Ga(1-s)^2 \Ga\left(\tfrac n2 +1-s\right)} \end{align*}
and
\begin{align*}
\int_0^\infty r^{2s-1} J_{n/2}(2\pi f r)&J_{n/2}(2\pi g r) \ud r =
\\&
= \frac 12 (\pi g)^{-2s} \left(\frac{f}{g}\right)^{n/2} \frac{\Ga\left(\tfrac n2 + s \right)}{\Ga(1-s)} \mathbf{F}\left(\tfrac n2 + s, s; \tfrac n2 + 1; \tfrac {f^2}{g^2}\right)\,.
\end{align*}
Putting it together results in
\begin{multline*}
\int_{\R^n} |\xi|^{2s} (\F \one_{[f,g]}(|\cdot|))(\xi)^2 \ud \xi =
\\
=\on{Vol}(S^{n-1}) \left( \frac{f^{-2s} + g^{-2s}}{2\pi^{2s}} \frac{\Ga\left(\tfrac n2 + s \right) \Ga(1-2s)}{\Ga(1-s)^2 \Ga\left(\tfrac n2 +1-s\right)} -\right.\\
\left.{}-\frac{g^{-2s}}{\pi^{2s}} \frac {f^{n/2}}{g^{n/2}} \frac{\Ga\left(\tfrac n2 + s \right)}{\Ga(1-s)} \mathbf{F}\left(\tfrac n2 + s, s; \tfrac n2 + 1; \tfrac {f^2}{g^2}\right) \right)\;.
\end{multline*}

In the limit $\ep\to 0$ we know from Lemma \ref{VectorfieldR} that $f^\ep(t) \to h(t)$ and $g^\ep(t) \to h(t)$ uniformly in $t$ on $[0,T]$ and hence $\tfrac {f^\ep(t)}{g^\ep(t)} \to 1$. For the regularized hypergeometric function $\mathbf{F}(a, b; c; d)$ at $d=1$ we have the identity \cite[(15.4.20)]{NIST2010}
\[
\mathbf{F}(a, b; c; 1) = \frac{\Ga(c-a-b)}{\Ga(c-a)\Ga(c-b)}\,,
\]
for $\on{Re}(c-a-b) > 0$. Applying the identity with $a=\tfrac n2+s$, $b=s$ and $c=\tfrac n2+1$  
we get
\[
\mathbf{F}\left(\tfrac n2 + s, s; \tfrac n2 + 1; 1\right) = \frac{\Ga(1-2s)}{\Ga(1-s)\Ga\left(\tfrac n2 +1-s\right)}\,.
\]
Using the continuity of the hypergeometric function it follows that 
\[
\int_{\R^n} |\xi|^{2s} \left(\F \one_{[f,g]}(|\cdot|))(\xi)\right)^2 \ud \xi \to 0\;,
\]
as $\ep \to 0$ and the convergence is uniform in $t$. This concludes the proof that
\[
\left\|\one_{[f^{\ep}(t), g^{\ep}(t)]}(|\cdot|)\right\|_{H^s({\R^n})} \on \to 0 \quad \text{as} \quad \ep \to 0 \quad \text{uniformly in } t\;,
\]
and hence we have established claim B for $\on{Diff}_c(\R^n)$.

To prove this result for an arbitrary manifold $M$ of bounded geometry 
we choose a partition of unity $(\ta_j)$ such that $\ta_0\equiv 1$ on some 
open subset $U\subset M$, where normal coordinates centred at $x_0\in M$ are defined.
If $\ph_{\R}$ is chosen with sufficiently small support, 
then the vector field $u^\ep_{\R^n}$ has support in $\exp_{x_0}(U)$ and we can define the vector field 
$u^\ep_M:=(\exp_{x_0}\i)^*u^\ep_{\R^n}$ on $M$. This vector field generates a path 
$\ph^\ep_M(t,\cdot) \in \Diff_0(M)$ with an endpoint 
$\ph^\ep_M(T,\cdot)=\ph_M(\cdot)$ that doesn't depend on $\ep$ with arbitrarily small $H^s$-length since
\begin{align*}
\on{Len}(\ph^\ep_M)&\leq C_1(\ta) \int_0^T\| u^\ep_M \|_{H^s(M,\ta)} \ud t=
C_1(\ta) \int_0^T\|\exp_{x_0}^*(\ta_0. u^\ep_M) \|_{H^s(\R^n)} \ud t
\\&=C_1(\ta) \int_0^T\|u^\ep_{\R^n} \|_{H^s(\R^n)} \ud t\;.
\end{align*}
Thus we can reduce the case of arbitrary manifolds to $\R^n$ and this concludes the proof.\
\end{proof}

\bibliographystyle{plain}

\begin{thebibliography}{10}

\bibitem{Bauer2011c_preprint}
Martin Bauer, Martins Bruveris, Philipp Harms, and Peter~W. Michor.
\newblock Geodesic distance for right invariant sobolev metrics of fractional
  order on the diffeomorphism group.
\newblock {\em To appear in Ann. Global Anal. Geom.}, 2012.

\bibitem{Bauer2012c}
Martin Bauer, Martins Bruveris, Philipp Harms, and Peter~W. Michor.
\newblock Vanishing geodesic distance for the {R}iemannian metric with geodesic
  equation the {K}d{V}-equation.
\newblock {\em Ann. Global Anal. Geom.}, 41(4):461--472, 2012.

\bibitem{Eichhorn2007}
J.~Eichhorn.
\newblock {\em Global analysis on open manifolds}.
\newblock Nova Science Publishers Inc., New York, 2007.

\bibitem{Eichhorn1991}
J{\"u}rgen Eichhorn.
\newblock The boundedness of connection coefficients and their derivatives.
\newblock {\em Math. Nachr.}, 152:145--158, 1991.

\bibitem{Greene1978}
R.~E. Greene.
\newblock Complete metrics of bounded curvature on noncompact manifolds.
\newblock {\em Arch. Math. (Basel)}, 31(1):89--95, 1978/79.

\bibitem{Kordyukov1991}
Yu.~A. Kordyukov.
\newblock {$L^p$}-theory of elliptic differential operators on manifolds of
  bounded geometry.
\newblock {\em Acta Appl. Math.}, 23(3):223--260, 1991.

\bibitem{Mather1974}
J.~N. Mather.
\newblock Commutators of diffeomorphisms.
\newblock {\em Comment. Math. Helv.}, 49:512--528, 1974.

\bibitem{Michor2012b_preprint}
Peter~W. Michor and David Mumford.
\newblock A zoo of diffeomorphism groups on $\mathbb {R}^n$.
\newblock 2012.
\newblock In preparation.

\bibitem{NIST2010}
Frank W.~J. Olver, Daniel~W. Lozier, Ronald~F. Boisvert, and Charles~W. Clark,
  editors.
\newblock {\em N{IST} handbook of mathematical functions}.
\newblock U.S. Department of Commerce National Institute of Standards and
  Technology, Washington, DC, 2010.

\bibitem{Shubin1992}
M.~A. Shubin.
\newblock Spectral theory of elliptic operators on noncompact manifolds.
\newblock {\em Ast\'erisque}, (207):5, 35--108, 1992.
\newblock M{\'e}thodes semi-classiques, Vol. 1 (Nantes, 1991).


\bibitem{Stein1971}
Elias~M. Stein and Guido Weiss.
\newblock {\em Introduction to {F}ourier analysis on {E}uclidean spaces}.
\newblock Princeton University Press, Princeton, N.J., 1971.
\newblock Princeton Mathematical Series, No. 32.

\bibitem{Thurston1974}
W. Thurston.
\newblock Foliations and groups of diffeomorphisms.
\newblock {\em Bull. Amer. Math. Soc.}, 80:304--307, 1974.

\bibitem{Triebel1992}
Hans Triebel.
\newblock {\em Theory of function spaces. {II}}, volume~84 of {\em Monographs
  in Mathematics}.
\newblock Birkh\"auser Verlag, Basel, 1992.

\end{thebibliography}

\end{document}